\documentclass[a4paper, 10.5pt]{article}
\usepackage{amsmath}
\usepackage{amssymb,esint}
\usepackage{amscd}
\usepackage{mathrsfs}
\usepackage{xspace}
\usepackage{fancyhdr}
\usepackage{xcolor}
\usepackage{verbatim}
\usepackage{cite}
\usepackage{appendix}
\usepackage{srcltx} 
\newcommand{\R}{\mathbb R}

\newtheorem{theorem}{Theorem}[section]

\newtheorem{lemma}[theorem]{Lemma}

\newtheorem{proposition}[theorem]{Proposition}

\newenvironment{proof}[1][Proof]{\textbf{#1.} }{\hfill\rule{0.5em}{0.5em}}
{\catcode`\@=11\global\let\AddToReset=\@addtoreset
	\AddToReset{equation}{section}
	
	\AddToReset{theorem}{section}

\begin{document}
		\title{A new proof of nonlinear Landau damping for the 3D Vlasov-Poisson system near Poisson equilibrium}

		\author{ {\bf Quoc-Hung Nguyen\thanks{E-mail address: qhnguyen@amss.ac.cn, Academy of Mathematics and Systems Science,
					Chinese Academy of Sciences,
					Beijing 100190, PR China} },
			{\bf Dongyi Wei\thanks{E-mail address: jnwdyi@pku.edu.cn, School of Mathematical Sciences, Peking University, Beijing 100871,  PR China}
}, 	{\bf Zhifei Zhang\thanks{E-mail address: zfzhang@math.pku.edu.cn, School of Mathematical Sciences, Peking University, Beijing 100871,  PR China}}
		}

		\date{}  
		\maketitle
				\begin{abstract}This paper investigates nonlinear Landau damping in the 3D Vlasov-Poisson (VP) system. We study the asymptotic stability of the Poisson equilibrium $\mu(v)=\frac{1}{\pi^2(1+|v|^2)^2}$
				under small perturbations. Building on the foundational work of Ionescu, Pausader, Wang, and Widmayer \cite{AIonescu2022}, we provide a streamlined proof of nonlinear Landau damping for the 3D unscreened VP system. Our analysis leverages sharp decay estimates, novel decomposition techniques to demonstrate the stabilization of the particle distribution and the decay of electric field. These results reveal the free transport-like behavior for the perturbed density $\rho(t,x)$, and enhance the understanding of Landau damping in an unconfined setting near stable equilibria.
			\end{abstract}
		\section{Introduction}
		This paper investigates nonlinear Landau damping for the 3D Vlasov-Poisson (VP) system. Specifically, we study the following system
		\begin{equation}\label{toymodela}
			\begin{cases}
				\partial_t \mathbf{f} + v\cdot \nabla_x  \mathbf{f}+E\cdot\nabla_v  \mathbf{f}=0,\\
				E=-\nabla_xU,  ~  -\Delta U ={\rho^*}-1,~~ {\rho^*}(t,x)=\int_{\mathbb{R}^3} \mathbf{f}(t,x,v)dv,\\
			\end{cases}
		\end{equation}
		in the whole space $x\in\mathbb{R}^3$, $v\in\mathbb{R}^3$. Here $\mathbf{f}=\mathbf{f}(t,x,v)\ge 0$ is the probability distribution of charged particles in plasma, $\rho^*(t,x)$ is the electric charge density, and  $E=E(t,x)$ is the electric field. The initial data $f_0(x,v)$ is close to the Poisson equilibrium  $\mu(v)=\frac{1}{\pi^2(1+|v|^2)^2}$.
	
	The dynamics and stability of the VP system have been extensively studied, with research primarily addressing global existence, regularity, and long-time behavior of the solution \cite{arsenev,  EHorst1982, CBardos1985,FBouchut1991,KPfaffelmoser1992, JSchaeffer1991,EHorst1993,RTGlassey1996,LionsPL1991,HJHwang2011,SHCHoi2011,JSmulevici2016,Xwang2018,Alonescwau2020,FlyOuPau,GriffinIacobelli2021a,GriffinIacobelli2021b}. 
Here we focus on the asymptotic stability of solutions $f(t,x,v)$ near the Poisson equilibrium. We write 
	$$\mathbf{f}(t,x,v)=\mu(v)+f(t,x,v),$$
	where $\mu(v)$ is the Poisson equilibrium and $\mathbf{f}(t=0,x,v)$ closes to $\mu(v)$.  The perturbed distribution $f(t,x,v)$ satisfies
\begin{equation}\label{eq2}
	\begin{cases}
		\partial_t f + v\cdot \nabla_x f +E\cdot\nabla_v \mu=-E\cdot\nabla_v f,\\
		E=\nabla_x\Delta^{-1}\rho,~~ \rho(t,x)=\int_{\mathbb{R}^d}f(t,x,v)dv,\\
		
		f|_{t=0}=f_0.
	\end{cases}
\end{equation}

Landau damping describes the stabilization of plasma distribution via phase mixing, where the electric field perturbation decays over time without dissipation. This phenomenon was rigorously established by Mouhot and Villani in the torus setting with analytic or Gevrey regularity \cite{CMouhot2011}. However, Bedrossian demonstrated that finite regularity solutions generally lack the same damping properties \cite{Bedrossiantunis}. Related mechanisms in the fluid dynamics, such as vorticity mixing by shear flows, have been extensively studied \cite{JbedIhes2015,ALonescu2020,Alonescu2020acta,Naderweiren,Chen}.

In the unbounded whole space $\mathbb{R}^d\times\mathbb{R}^d$
for $d\geq 3$, nonlinear Landau damping was established for the screened VP system \cite{JBedrossian2018,HanKwanD2021,HNX1}. Recently, the first author and collaborators proved nonlinear Landau damping for the 2D screened VP system \cite{HNX2}. For the unscreened case, Ionescu, Pausader, Wang, and Widmayer \cite{AIonescu2022} rigorously demonstrated the decay of electric field and stabilization of the particle distribution in $\mathbb{R}^3$ near the Poisson equilibrium. Their key innovation was the decomposition of the electric field into static and oscillatory components with distinct decay rates. However, general case for the unscreened system remains open.

Building on the work \cite{AIonescu2022}, this paper introduces a concise and novel proof of nonlinear Landau damping for the 3D VP system near the Poisson equilibrium. Our approach simplifies the analysis while maintaining its depth and rigor. Roughly speaking, our main results include as follows.

\begin{itemize}
	\item We establish precise bounds for the density $\rho$ showing its stabilization and free transport-like behavior.
	\item
	Streamlined analysis: by leveraging new decomposition techniques, we simplify the derivation of global estimates.
	\item
	Asymptotic stability: we prove the long-time stability of solutions, demonstrating the stabilization of the distribution function $f(t,x,v)$ and the decay of electric field.
\end{itemize}

We state the main theorem, which demonstrates the global existence and asymptotic stability of the perturbed system.

	\begin{theorem} \label{mainthm} Let $0<\kappa_0\ll 1$. Assume that  $\iint f_0(x,v)dxdv=0$.  Then there exists $\epsilon_0>0$ such that if 
		\begin{equation}
			[f_0]=\sum_{j=0,1,2}\sup_{x,v}\langle x,v\rangle^{10}|\nabla_{x,v}^jf_0(x,v)|
			\leq \epsilon_0,
		\end{equation}
		then the VP system \eqref{eq2} has global unique solution $f$ satisfying 
		\begin{align}
			|	\nabla\Delta^{-1}\rho(t,x)|+\langle t,x\rangle	|	\nabla^2\Delta^{-1}\rho(t,x)|+\langle t,x\rangle^{1+\kappa_0}	|	\nabla^3\Delta^{-1}\rho(t,x)|\leq\langle t,x\rangle^{-3+\kappa_0} [f_0],
		\end{align}
		for any $t>0$ and $x\in \mathbb{R}^3$. Moreover, there is a  function $f_\infty\in C^{0,1}_{x,v}$ satisfying 
		\begin{align}\label{po}
		\langle v\rangle^{5}	|f(t,x+tv,v)-f_\infty(x,v)|\lesssim [f_0]\int_{t}^{\infty}\tau \langle \tau,x+(\tau-t)v\rangle^{-3+\kappa_0}d\tau,
		\end{align}
	for any $t>0$ and $x\in \mathbb{R}^3$.
	\end{theorem}

			By considering the associated characteristics $\left(X_{s,t}(x,v),V_{s,t}(x,v)\right)$, we express the solution $f$ explicitly as
		\begin{equation}\label{0}
			f(t,x,v)=f_0(X_{0,t}(x,v),V_{0,t}(x,v))-\int_{0}^{t}E(s,X_{s,t}(x,v))\cdot\nabla_v\mu(V_{s,t}(x,v))ds.
		\end{equation}
		We introduce the following key decompositions
		\begin{align*}
			&	\mathcal{I}(\rho)(t,x,v)=f_0(X_{0,t}(x,v),V_{0,t}(x,v)),\\&
			\mathcal{R}(\rho)(t,x,v)=\int_0^tE(s,x-(t-s)v)\cdot\nabla_v\mu(v)-E(s,X_{s,t}(x,v))\cdot\nabla_v\mu(V_{s,t}(x,v)))dvds,
		\end{align*}
		which lead to important macroscopic quantities such as
		\begin{align}\label{b1}
			&\mathbf{R}_0(t,x)=\int \mathcal{I}(\rho)(t,x,v)+\mathcal{R}(\rho)(t,x,v) dv,\\&\label{b2}
			\mathbf{R}_1(t,x)=\int( \mathcal{I}(\rho)(t,x,v)+\mathcal{R}(\rho)(t,x,v) ) vdv,\\&\label{b3}
			\mathbf{R}_2(t,x)=\int( \mathcal{I}(\rho)(t,x,v)+\mathcal{R}(\rho)(t,x,v) ) v\otimes vdv.
		\end{align}
		These quantities obey the following conservation laws and structural equations, which are critical to the analysis (see \eqref{Z1}):
			\begin{align}\label{Z0}
				\partial_t\mathbf{R}_0+\text{div}\mathbf{\mathbf{R}_1}=0,~ \partial_t\mathbf{\mathbf{R}_1}+\text{div}\mathbf{R}_2=\operatorname{div}(E(t,x)\otimes E(t,x))-\frac{1}{2}\nabla(| E(t,x)|^2).
			\end{align}	
Using the formula for the density $\rho$, as detailed in  \eqref{a2}, and leveraging the assumption $$\iint f_0(x,v)dxdv=0,$$
we systematically establish global estimates for $\rho$. These results align with and extend the methodologies developed for the screened Vlasov-Poisson system in \cite{HanKwanD2021}, \cite{HNX1} and \cite{HNX2}.

		\section{Solving the density equation and Bootstrap setup}
		This section focuses on solving the density equation $\rho(t,x)$ using the macroscopic moments $\mathbf{R}_0,\mathbf{R}_1,\mathbf{R}_2$. Following the approach in \cite{HanKwanD2021} and \cite{AIonescu2022}, we derive a decomposition of $\rho(t,x)$ into oscillatory and residual components, which allows us to establish precise decay estimates for the bootstrap argument.
		
		Starting from the solution representation \eqref{0}, the density  $\rho(t,x)$ satisfies
	\begin{equation}\label{equation rho}
		\rho=-\int_{0}^{t}\sin(s)G(s)\star\mathbf{R}_0(\rho)(t-s)(x) ds+\mathbf{R}_0(\rho)(t,x),
	\end{equation}
	where  
	\begin{align*}
		\hat G(s,\xi)=e^{-s|\xi|}, \quad G(s,x)=\frac{1}{\pi^2}\frac{ s}{(s^2+|x|^2)^{2}},\quad \forall s>0.
	\end{align*}
	As in \cite{HanKwanD2021},  the  characteristics $\left(X_{s,t}(x,v),V_{s,t}(x,v)\right)$ of \eqref{eq2} can be given  for any $0\leq s\leq t< \infty $, 
	\begin{align}
		\label{z4}
	&	X_{s,t}(x,v)=x-(t-s)v+Y_{s,t}(x-vt,v),~~V_{s,t}(x,v)=v+W_{s,t}(x-vt,v),\\&
	\label{definition Y W}
	Y_{s,t}(x,v)=\int_{s}^{t}(\tau-s)E(\tau,x+\tau v+Y_{\tau,t}(x,v))d\tau,\quad
	W_{s,t}(x,v)=-\int_{s}^{t}E(\tau,x+\tau v+Y_{\tau,t}(x,v))d\tau.
	\end{align}
\subsection{Decomposition of $\rho(t,x)$}
We define the quantity
\begin{equation}
	F(t,x,v):=f(t,x,v)+\int_0^tE(s,x-(t-s)v)\cdot\nabla\mu(v)ds.
\end{equation}
This can be split into two components
 $$F(t,x,v)=\mathcal{I}(\rho)(t,x,v)+\mathcal{R}(\rho)(t,x,v),$$ 
 where 
  \begin{align*}& \mathcal{I}(\rho)(t,x,v) = f_0(X_{0,t}(x,v), V_{0,t}(x,v)),\\& \mathcal{R}(\rho)(t,x,v) = \int_{0}^{t} \left(E(s, x - (t-s)v) \cdot \nabla_v \mu(v) - E(s, X_{s,t}(x,v)) \cdot \nabla_v \mu(V_{s,t}(x,v))\right)ds. \end{align*}
Then we have 
		\begin{align*}
	\partial_tF+v\cdot\nabla_xF=-E\cdot\nabla_vf.
\end{align*}
The macroscopic moments $\mathbf{R}_0,\mathbf{R}_1,\mathbf{R}_2$  defined in \eqref{b1}, \eqref{b2} and \eqref{b3}
satisfy the following coupled system 
\begin{align}
	\partial_t\mathbf{R}_0+\text{div}\mathbf{R}_1=0,\quad \partial_t\mathbf{R}_1+\text{div}\mathbf{R}_2=
	E(t,x)\rho(t,x)
\end{align}
Using the relationship for the electric field and charge density
$$E(t,x)\rho(t,x)=\operatorname{div}(E(t,x)\otimes E(t,x))-\frac{1}{2}\nabla(| E(t,x)|^2),$$ 
the system can be rewritten as
\begin{align}\label{Z1}
	\partial_t\mathbf{R}_0+\text{div}\mathbf{R}_1=0,\quad \partial_t\mathbf{R}_1+\text{div}\mathbf{R}_2=\operatorname{div}(E(t,x)\otimes E(t,x))-\frac{1}{2}\nabla(| E(t,x)|^2).
\end{align}
This formulation encapsulates the conservation laws and structural equations central to the analysis. The system not only describes the evolution of density and momentum, but also captures the influence of nonlinear interactions through the electric field term. It can be easily checked that 
	\begin{align*}
		\mathbf{R}_0(\rho)(0,x)=\tilde{\rho}_{0,0}(x), \quad \mathbf{R}_1(\rho)(0,x)= \tilde{\rho}_{1,0}(x),\quad 
		\mathbf{R}_2(\rho)(0,x)= \tilde{\rho}_{2,0}(x),
	\end{align*}
where
$\tilde{\rho}_{j,0}(x)=\int v^{\otimes^j}f_0(x,v)dv.$

\subsection{Singular and Residual components of $\rho(t,x)$}

	To analyze $\rho(t,x)$, we introduce a smooth cut-off function $\chi$ such that $\chi=1$ in $B_1(0)$ and $\chi=0$ in $\mathbb{R}^2\backslash B_2(0).$\\
	Set $\phi=\hat \chi$. We define 
	\begin{equation}
		f_{<}=f\star\phi,~~~	f_{>}=f-f\star\phi.
	\end{equation}
Using this decomposition, the density equation becomes
	\begin{equation*}
		\rho(t,x)=-\int_{0}^{t}\sin(s)G_{>}(s)\star \mathbf{R}_0(\rho)(t-s)(x)ds+\mathbf{R}_0(\rho)(t,x)-\int_{0}^{t}\sin(s)G_{<}(s)\star \mathbf{R}_0(\rho)(t-s)(x)ds.
	\end{equation*}
	By integration by parts twice,  and using \eqref{Z1} and  the initial conditions
	$$\mathbf{R}_0(\rho)(0,x)=\tilde{\rho}_{0,0}(x), \quad \partial_t\mathbf{R}_0(\rho)(0,x)= -\operatorname{div}\tilde{\rho}_{1,0}(x),$$ 
	we derive
	\begin{align}\label{a2}
		\rho(t,x)= \cos(t)\rho_{sing}^1(t,x)+\sin(t)\rho_{sing}^2(t,x)+	\rho_{re}(t,x),
	\end{align}
where 
\begin{align*}
&	\rho_{sing}^1(t,x)=G_{<}(t)\star \tilde\rho_{0,0}(x),\\&
\rho_{sing}^2(t,x)=-(\partial_tG_{<}(t)\star \tilde\rho_{0,0}(x)+G_{<}(t)\star \operatorname{div} \tilde\rho_{0,1}(x)),\\&
	\rho_{re}(t,x)=	\mathbf{R}_{0,>}(\rho)(t,x)-\int_{0}^{t}\sin(s)G_{>}(s)\star \mathbf{R}_0(\rho)(t-s)(y)ds		+\int_{0}^{t}\sin(s)\partial_{s}^2G_{<}(s)\star \mathbf{R}_0(\rho)(t-s)(x)ds	\\&\qquad+2\int_{0}^{t}\sin(s)\nabla\partial_{s} G_{<}(s)\star\mathbf{R}_{1}(\rho)(t-s)(x)ds+	\int_{0}^{t}\sin(s)\nabla^{\otimes^2} G_{<}(s)\star\mathbf{R}_2(\rho)(t-s)(x)ds\\&\qquad
	-	\int_{0}^{t}\sin(s)\nabla^2 G_{<}(s)\star(E\otimes E)(t-s)(x)ds	+\frac{1}{2}	\int_{0}^{t}\sin(s)\Delta G_{<}(s)\star(|E|^2)(t-s)(x)ds.
\end{align*}

Using the Fourier transform properties of $G_{<}$ and $G_>$, we can obtain the following decay estimates of the kernel
		\begin{align}\label{esG1}
		&	|\partial_{t}^{j_1}\nabla_x^{j_2}\nabla (-\Delta)^{-1}G_{<}(t,x)|\lesssim \frac{1}{\langle x,t\rangle^{2+j_1+j_2}},\\
		&	|\partial_{t}^{j_1}\nabla_x^{j_2}G_{<}(t,x)|\lesssim \frac{\langle t\rangle \mathbf{1}_{j_1=0}}{\langle x,t\rangle^{4+j_2}}+\frac{\mathbf{1}_{j_1\geq 1}}{\langle x,t\rangle^{3+j_1+j_2}},\\&
		|\partial_{t}^{j_1}\nabla_x^{j_2}\nabla (-\Delta)^{-1}G_{>}(t,x)|\lesssim \frac{t\mathbf{1}_{j_1=0}}{(t+|x|)^{4+j_2}\langle x,t\rangle^{7}}+\frac{\mathbf{1}_{j_1\geq 1}}{(t+|x|)^{3+j_1+j_2}\langle x,t\rangle^{7}}.\label{esG2}
	\end{align}
These estimates will play a critical role in bounding $\rho_{re}(t,x)$ and completing the bootstrap argument.

\section{Norms and the bootstrap proposition.}
We now introduce the key norms used in our analysis and state the bootstrap proposition, which plays a central role in establishing the global existence and stability of solutions.

The following norms are used to quantify the behavior of the density and the moments
		\begin{align*}
		&	\|g\|_{1,T}=\sup_{t\in [0,T]}\|\langle t,x\rangle^{3-\kappa_0}\nabla\Delta^{-1}g\|_{L^{\infty}}+\|\langle t,x\rangle^{4-\kappa_0}\nabla^2\Delta^{-1}g\|_{L^{\infty}}+\|\langle t,x\rangle^{4}\nabla^3\Delta^{-1}g\|_{L^{\infty}},\\&
				\|g\|_{2,T}=\sup_{t\in [0,T]}\|\langle t,x\rangle^{4}\nabla g\|_{L^{\infty}}+\sup_{t\in [0,T]}\|\langle t,x\rangle^{3} g\|_{L^{\infty}},
		\end{align*}
	where $\langle t,x\rangle=\sqrt{1+|x|^2+t^2}$ and $\kappa_0>0$ is a small parameter.
	
	Under the initial data assumptions, the singular parts of the density  $\rho_{sing}^1,\rho_{sing}^2$ satisfy the estimate
\begin{align}\label{a1}
\sum_{j=1,2}|\partial_{t}^{n_2}\nabla^{1+n_1} \Delta^{-1}	\rho_{sing}^j(t,x)|\lesssim \frac{	[f_0]}{\langle t,x\rangle^{3+n_1+n_2}}, ~~~\forall n_1+n_2\leq 4.
\end{align}
Using the decomposition \eqref{a2} and the estimate \eqref{a8}, we can bound
\begin{equation}\label{a7}
	\|\rho\|_{1,T}\lesssim [f_0]+\|( \mathbf{R}_0(\rho),\mathbf{R}_1(\rho), \mathbf{R}_2(\rho))\|_{2,T}+	\|\rho\|_{1,T}^2
\end{equation}
for any $T>0$.

Now we state the main bootstrap proposition, which provides the control over $\|\rho\|_{1,T}$.

	\begin{proposition} \label{propoglobal3d}  There exists a $\epsilon>0$ such that if$$	M:=\|\rho\|_{1,T}+[f_0]\leq \epsilon\ll 1$$ for some $T>0$, then 
	\begin{align}\label{A}
		\|\rho\|_{1,T}\lesssim [f_0]+\|\rho\|_{1,T}^2.
	\end{align}
\end{proposition}

This proposition follows as a consequence of \eqref{a7} and the proposition below.

		\begin{proposition} \label{propoglobal3db} 	Under the same assumptions, we have
			\begin{align}\label{A0}
			\|( \mathbf{R}_0(\rho),  \mathbf{R}_1(\rho),  \mathbf{R}_2(\rho))\|_{2,T}\lesssim [f_0]+\|\rho\|_{1,T}^2.
			\end{align}
		\end{proposition}

	With Proposition \ref{propoglobal3db} established, we can deduce the global existence of solutions to the perturbed system \eqref{eq2}, with the density $\rho(t,x)$ satisfying
$$
		\|\rho\|_{1,\infty}\lesssim [f_0].$$
	This conclusion follows from a standard bootstrap argument (see \cite{AIonescu2022, HNX2}). Additionally, using the argument in Proposition 8.1 of \cite{AIonescu2022}, we can establish the existence of a function $f_\infty\in C^{0,1}_{x,v}$ such that
	\begin{align}
	\langle v\rangle^{5}	|f(t,x+tv,v)-f_\infty(x,v)|\lesssim [f_0]\int_{t}^{\infty}\tau \langle \tau,x+(\tau-t)v\rangle^{-3+\kappa_0}d\tau	.
\end{align}
for any $t>0$ and $x\in \mathbb{R}^3$. This completes the proof of Theorem \ref{mainthm}. \medskip

\begin{proof}[Proof of Proposition \ref{propoglobal3db}]  {\bf Step 1.} \textit{Estimates for $Y_{s,t}$ and $W_{s,t}$}.\smallskip

 Following the approach in \cite{HNX1}, we can bound  $Y_{s,t},W_{s,t}$ as follows 
		\begin{align*}
			&\|\nabla_{x}^{n_1}\nabla_{v}^{n_2}Y_{s,t}\|_{L^\infty_{x,v}}\lesssim\left( \mathbf{1}_{n\leq 1}\langle s\rangle^{-1-n_1+\kappa_0}+\mathbf{1}_{n= 2}\langle s\rangle^{-2+n_2}\right)M,\\&
		\|\nabla_{x}^{n_1}\nabla_{v}^{n_2}W_{s,t}\|_{L^\infty_{x,v}}\lesssim\left(\mathbf{1}_{n\leq 1}\langle s\rangle^{-2-n_1+\kappa_0}+\mathbf{1}_{n= 2}\langle s\rangle^{-3+n_2}\right)M,
		\end{align*}
	for any $n:=n_1+n_2\leq  2,  n_2=0,1$, where $M\ll 1$. Using these bounds, we can obtain the following pointwise estimates
		\begin{align}\label{a9}
			&|\nabla_{x}^{n_1}\nabla_{v}^{n_2}Y_{s,t}(x,v)|\lesssim \int_{s}^{t}(\tau-s)\tau^{n_2}\left[\mathbf{1}_{n\leq 1}\langle \tau,x+\tau v\rangle^{-3-n+\kappa_0}+\mathbf{1}_{n= 2}\langle \tau,x+\tau v\rangle^{-4}\right]d\tau M,\\&
			|\nabla_{x}^{n_1}\nabla_{v}^{n_2}W_{s,t}(x,v)|\lesssim\int_{s}^{t}\tau^{n_2}\left[\mathbf{1}_{n\leq 1}\langle \tau,x+\tau v\rangle^{-3-n+\kappa_0}+\mathbf{1}_{n= 2}\langle \tau,x+\tau v\rangle^{-4}\right]d\tau M,\label{a10}
		\end{align}
	for any $n=n_1+n_2\leq  2, n_2=0,1.$\smallskip
	
	{\bf Step 2.} \textit{Properties of the map $\Psi_{s,t}$}.\smallskip

	Define the map $\Psi_{s,t}(x,v)$ such that
 	\begin{align}&
 	\label{X(x,Psi)}
 	X_{s,t}(x,\Psi_{s,t}(x,v))=x-(t-s)v,\\
 	\label{z16a}
 	&\Psi_{s,t}(x,v)-v=-\Phi_{s,t}(x,\Psi_{s,t}(x,v)),
 \end{align}
 where 
 \begin{equation}
 	\Phi_{s,t}(x,v)=-\frac{1}{t-s}\int_{s}^{t}(\tau-s)E(\tau,x-(t-\tau)v+Y_{\tau,t}(x-vt,v))d\tau.\label{z16b}
 \end{equation}
	We can bound  $|\Psi_{s,t}(x,v)-v|$ as follows
	\begin{align*}
	|	\Psi_{s,t}(x,v)-v|&\lesssim \frac{1}{t-s}\int_{s}^{t}(\tau-s)\langle \tau,x-(t-\tau)\Psi_{s,t}+Y_{\tau,t}(x-\Psi_{s,t}t,\Psi_{s,t})\rangle^{-3+\kappa_0}d\tau M\\
	&\lesssim \frac{1}{t-s}\int_{s}^{t}(\tau-s)\langle \tau,x-(t-\tau)\Psi_{s,t}\rangle^{-3+\kappa_0}d\tau M.
	\end{align*}
This implies that $|\Psi_{s,t}(x,v)-v|\lesssim \langle t\rangle^{-1}\langle s\rangle^{-1+\kappa_0} M$ and 
	\begin{align}\label{a13}
	|	\Psi_{s,t}(x,v)-v|
	\lesssim\int_{s}^{t} \frac{\tau-s}{t-s}\langle \tau,x-(t-\tau)v\rangle^{-3+\kappa_0}d\tau M.
\end{align}
 For $\Psi_{s,t}=\Psi_{s,t}(x,v)$, we have
	\begin{align*}
	&|\nabla_{v}\Psi_{s,t}(x,v)-I|\leq \|\nabla_v	\Psi_{s,t}\|_{L^\infty_{x,v}} |(\nabla_{v}\Phi_{s,t})(x,\Psi_{s,t}(x,v))|\\ &\lesssim  \|\nabla_v	\Psi_{s,t}\|_{L^\infty_{x,v}} \frac{1}{t-s}\int_{s}^{t}(\tau-s)(t-\tau+1)\langle \tau,x-(t-\tau)\Psi_{s,t}+Y_{\tau,t}(x-\Psi_{s,t}t,\Psi_{s,t})\rangle^{-4+\kappa_0}d\tau M
	\\ &\lesssim  \|\nabla_v	\Psi_{s,t}\|_{L^\infty_{x,v}} \frac{1}{t-s}\int_{s}^{t}(\tau-s)(t-\tau+1)\langle \tau,x-(t-\tau)v\rangle^{-4+\kappa_0}d\tau M.
\end{align*}
Thus, $\|\nabla_v	\Psi_{s,t}\|_{L^\infty_{x,v}}\leq 4$ and 
	\begin{align}\label{a15}
	|	\det(\nabla_{v}\Psi_{s,t})(x,v)-1|\lesssim  \int_{s}^{t}\frac{(\tau-s)(t-\tau+1)}{t-s}\langle \tau,x-(t-\tau)v\rangle^{-4+\kappa_0}d\tau M.
\end{align}

{\bf Step 3}. \textit{Estimates  of $\mathcal{I}_j$ and $\mathcal{R}_j$.}\smallskip

We define
	\begin{align*}
		&\mathcal{I}_j(\rho)(t,x)=\int \mathcal{I}(\rho)(t,x,v)v^{\otimes^j}dv,~~~
		\mathcal{R}_j(\rho)(t,x)\!\!=\!\!\int_0^t\!\!\!\int  \mathcal{R}(\rho)(t,x,v)v^{\otimes^j}dvds.
	\end{align*}
Thus, $\mathbf{R}_j(\rho)=	\mathcal{I}_j(\rho)+	\mathcal{R}_j(\rho)$.\smallskip

\textbf{(a) Bound for $\mathcal{I}_j(\rho)$}. We can derive that for $j\leq 2$ and $m=0,1$,
		\begin{align*}
		\langle t\rangle^m|\nabla^m_x\mathcal{I}_j(\rho)(t,x)|\lesssim M \int_{\R^3}\langle x-tv, v\rangle^{-7}dv [f_0].
		\end{align*}
By \eqref{a3}, we obtain 
		\begin{align}\label{a16}
			\sum_{j=0,1,2}	\|\mathcal{I}_j(\rho)\|_{2,T}\lesssim [f_0].
		\end{align}
		
	\textbf{(b) Bound for $\mathcal{R}_j(\rho)$}. For $j=0,1,2$, $\mathcal{R}_j(\rho)$ can be expressed as
	\begin{align}
		\mathcal{R}_j(\rho)(t,x)&=\int_0^t\!\!\!\int_{\R^3}\!\!(E(s,x-(t-s)v)\!\cdot\!\nabla_v\mu(v)\!-\!E(s,X_{s,t}(x,v))\!\cdot\!\nabla_v\mu(V_{s,t}(x,v)))v^{\otimes^j}dvds\nonumber\\
		&:=\mathcal{T}[E,\nabla\mu,v^{\otimes^j}](t,x).
	\end{align}
We have 
\begin{align*}
		|\mathcal{R}_j(\rho)(t,x)|\lesssim& \int_0^t\int|E(s,x-(t-s)v)|\langle v\rangle^{-4} |W_{s,t}(x-tv,x)|dv ds\\&+
		\int_0^t\int|E(s,x-(t-s)v)-E(s,X_{s,t}(x,v))|\langle v\rangle^{-3}dv ds.
\end{align*}
By the definition of $\|\rho\|_{1,T}$, we get by \eqref{a9} and  \eqref{a10}  that
\begin{align}\nonumber
	|\mathcal{R}_j(\rho)(t,x)|&\lesssim \int_0^t\int_{s}^{t}\int_{\R^3}\langle s,x-(t-s)v\rangle^{-3+\kappa_0}  \langle \tau,x-(t-\tau) v\rangle^{-3+\kappa_0} \frac{dv d\tau ds}{\langle v\rangle^{4}}\|\rho\|_{1,T}M\\&\quad\nonumber+
	\int_0^t\int_{s}^{t}(\tau-s)\int_{\R^3} \langle s,x-(t-s)v\rangle^{-4+\kappa_0}\langle \tau,x-(t-\tau) v\rangle^{-3+\kappa_0}\frac{dv ds}{\langle v\rangle^{3}}\|\rho\|_{1,T}M
	\\&:= A_1(t,x)\|\rho\|_{1,T}M+A_2(t,x)\|\rho\|_{1,T}M.\label{a12}
\end{align}
By \eqref{a4}, we obtain 
\begin{align}\label{a5}
	|\mathcal{R}_j(\rho)(t,x)|\lesssim \langle x,t\rangle^{-3}\|\rho\|_{1,T}.
\end{align}

For $t\in [0,1]$, we have 
\begin{align*}
	|\nabla_{x}\mathcal{R}_j(\rho)(t,x)|\lesssim& \int_0^t\int|(\nabla_{x}E)(s,x-(t-s)v)\cdot\nabla_v\mu(v)-(\nabla_{x}E)(s,X_{s,t}(x,v))\nabla_v\mu(V_{s,t}(x,v))|\langle v\rangle^2dvds\\&+
\int_0^t\int|(\nabla_{x}E)(s,X_{s,t}(x,v))|\langle v\rangle^{-3}|\nabla Y_{s,t}(x-tv,v)|dvds
\\&+
\int_0^t\int|E(s,X_{s,t}(x,v))|\langle v\rangle^{-4}|\nabla W_{s,t}(x-tv,v)|dvds.
\end{align*}
By the definition of $\|\rho\|_{1,T}$, we get by \eqref{a9} and  \eqref{a10}  that 
\begin{align*}
	|\nabla_{x}\mathcal{R}_j(\rho)(t,x)|\lesssim& \int_0^t\int_{s}^{t}\int_{\R^3}\langle x-(t-s)v\rangle^{-4+\kappa_0}\langle v\rangle^{-4}  \langle x-(t-\tau) v\rangle^{-3+\kappa_0} dv d\tau ds\|\rho\|_{1,T}M\\&+
	\int_0^t\int_{s}^{t}(\tau-s)\int_{\R^3} \langle x-(t-s)v\rangle^{-4}\langle x-(t-\tau) v\rangle^{-3+\kappa_0}\langle v\rangle^{-3}dvd\tau  ds\|\rho\|_{1,T}M
	\\&+
	\int_0^t\int_{s}^{t}(\tau-s)\int_{\R^3} \langle x-(t-s)v\rangle^{-4+\kappa_0}\langle x-(t-\tau) v\rangle^{-4+\kappa_0}\langle v\rangle^{-3}dvd\tau  ds\|\rho\|_{1,T}M
		\\&+
	\int_0^t\int_{s}^{t}\int_{\R^3} \langle x-(t-s)v\rangle^{-3+\kappa_0}\langle x-(t-\tau) v\rangle^{-4+\kappa_0}\langle v\rangle^{-4}dv d\tau ds\|\rho\|_{1,T}M.
\end{align*}
Then we divide the integration of $v$ into $\int\mathbf{1}_{\langle x\rangle\geq 4(t-s)\langle v \rangle}+\int\mathbf{1}_{\langle x\rangle\leq 4(t-s)\langle v \rangle}$ and easily get 
\begin{equation}\label{a11}
		|\nabla_{x}\mathcal{R}_j(\rho)(t,x)|\lesssim \langle x\rangle^{-4}\|\rho\|_{1,T}M.
\end{equation}

For $t\geq 1$,  as in \cite{HNX1}, we have  
\begin{align*}
t	\partial_{x_i}\mathcal{T}[E,\nabla\mu,v^{\otimes^j}](t,x)&=\mathcal{T}[	\tilde{F}_i,\nabla\mu,v^{\otimes^j}](t,x)+ 	\mathcal{T}[E,\partial_{i}\nabla\mu,v^{\otimes^j}](t,x)\\&+	\mathcal{T}[E,\nabla\mu,\partial_{i}v^{\otimes^j}](t,x)+\mathcal{T}^{1,i}[E,\nabla\mu,v^{\otimes^j}](t,x)+\mathcal{T}^{2,i}[E,\nabla\mu,v^{\otimes^j}](t,x),~~i=1,2,3,
\end{align*}
where $
\tilde{F}_i(s,.):=s\partial_{x_i}E(s,.)$ and 
\begin{align*}
	\mathcal{T}^{1,i}[E,\nabla\mu,v^{\otimes^j}](t,x)&:=-\int_{0}^{t}\int 
	E\left(s,X_{s,t}\right)(\nabla^2\mu)\left(V_{s,t}\right)\cdot(\partial_{v_i}W_{s,t})(x-tv,v) v^{\otimes^j}
	dvds,\\
	\mathcal{T}^{2,i}[E,\nabla\mu,v^{\otimes^j}](t,x)&:=-\int_{0}^{t}\int 
	\nabla	E\left(s,X_{s,t}\right)\cdot(\partial_{v_i}Y_{s,t})(x-tv,v) \nabla\mu\left(V_{s,t}\right) v^{\otimes^j}
	dvds.
\end{align*}
As in \eqref{a12}, we have 
\begin{align*}
	&\frac{|\mathcal{T}^{1,i}[E,\nabla\mu,v^{\otimes^j}](t,x)|}{\|\rho\|_{1,T}M}\lesssim  \int_0^t\int_{s}^{t}\int\tau \langle s,x-(t-s)v\rangle^{-3+\kappa_0}\langle \tau,x-(t-\tau) v\rangle^{-4+\kappa_0}\frac{dv d\tau ds}{\langle v\rangle^{4}}:= A_3(t,x),\\&	\frac{|\mathcal{T}^{2,i}[E,\nabla\mu,v^{\otimes^j}](t,x)|}{\|\rho\|_{1,T}M}\lesssim  	\int_0^t\int_{s}^{t}\int_{\R^3}(\tau-s)\tau \langle s,x-(t-s)v\rangle^{-4+\kappa_0}\langle \tau,x-(t-\tau) v\rangle^{-4+\kappa_0}\frac{dv d\tau ds}{\langle v\rangle^{3}}:= A_4(t,x),\\
&	\frac{1}{\|\rho\|_{1,T}M}\left(	|\mathcal{T}[E,\partial_{i}\nabla\mu,v^{\otimes^j}](t,x)|+	|	\mathcal{T}[E,\nabla\mu,\partial_{i}v^{\otimes^j}](t,x)|\right)\\&\quad\quad\quad\quad\lesssim   \int_0^t\int_{s}^{t}\int_{\R^3}\langle s,x-(t-s)v\rangle^{-3+\kappa_0} \langle \tau,x-(t-\tau) v\rangle^{-3+\kappa_0} \frac{dv d\tau ds}{\langle v\rangle^{5}}\\&\quad\quad\quad\quad\quad+ 
	\int_0^t\int_{s}^{t}(\tau-s)\int_{\R^3} \langle s,x-(t-s)v\rangle^{-4+\kappa_0}\langle \tau,x-(t-\tau) v\rangle^{-3+\kappa_0}\frac{dv d\tau ds}{\langle v\rangle^{4}}
	:= A_5(t,x)+A_6(t,x),
\end{align*}
and 
\begin{align*}
	&
	\frac{|	\mathcal{T}[	\tilde{F}_i,\nabla\mu,v^{\otimes^j}](t,x)|}{\|\rho\|_{1,T}M}\lesssim  \int_0^t\int_{s}^{t}s\int_{\R^3}\langle s,x-(t-s)v\rangle^{-4+\kappa_0}  \langle \tau,x-(t-\tau) v\rangle^{-3+\kappa_0} \frac{dv d\tau ds}{\langle v\rangle^{4}}\\&\quad\quad\quad\quad+  
	\int_0^t\int_{s}^{t}(\tau-s)s\int_{\R^3} \langle s,x-(t-s)v\rangle^{-4}\langle \tau,x-(t-\tau) v\rangle^{-3+\kappa_0}\frac{dv d\tau ds}{\langle v\rangle^{3}}.
\end{align*}
The second term does not decay fast enough. To address the insufficient decay of this term, we apply the variable change strategy described in \cite{HNX2}. We rewrite $	\mathcal{T}[	\tilde{F}_i,\nabla\mu,v^{\otimes^j}](t,x)$  using the map $\Psi_{s,t}$
which satisfies $$X_{s,t}(x,\Psi_{s,t}(x,v))=x-(t-s)v.$$
Using this map, the term $	\mathcal{T}[	\tilde{F}_i,\nabla\mu,v^{\otimes^j}](t,x)$ can be bounded as

\begin{align*}
	&|\mathcal{T}[	\tilde{F}_i,\nabla\mu,v^{\otimes^j}](t,x)|\\ &\lesssim \int_{0}^{t}s\int |(\nabla E)(s,x-(t-s)v)| \left|(\Psi_{s,t}(x,v))^{\otimes^j}\nabla\mu(V_{s,t}(x,\Psi_{s,t}))\det(\nabla_v\Psi_{s,t})	-v^{\otimes^j}\nabla\mu(v)\right|dvds.
\end{align*}
The difference term inside the integral can be decomposed and estimated as follows
\begin{align*}
&	\left|(\Psi_{s,t}(x,v))^{\otimes^j}\nabla\mu(V_{s,t}(x,\Psi_{s,t}))\det(\nabla_v\Psi_{s,t})	-v^{\otimes^j}\nabla\mu(v)\right|\lesssim\langle v\rangle^{-3}|\det(\nabla_v\Psi_{s,t})(x,v)-1|\\&\qquad+
\langle v\rangle^{-4}\left(|\Psi_{s,t}(x,v)-v|+|W_{s,t}(x-t\Psi_{s,t},\Psi_{s,t}(x,v))|\right).
\end{align*}
Then by \eqref{a13}, \eqref{a15} and \eqref{a10}, we get 
\begin{align*}
	&	\left|(\Psi_{s,t}(x,v))^{\otimes^j}\nabla\mu(V_{s,t}(x,\Psi_{s,t}))\det(\nabla_v\Psi_{s,t})	-v^{\otimes^j}\nabla\mu(v)\right|\\
&\lesssim \langle v\rangle^{-3}\int_{s}^{t}\frac{(\tau-s)(t-\tau+1)}{t-s}\langle \tau,x-(t-\tau)v\rangle^{-4+\kappa_0}d\tau M+
\langle v\rangle^{-4}\int_{s}^{t} \langle \tau,x-(t-\tau)v\rangle^{-3+\kappa_0}d\tau M.
\end{align*}
Substituting the above estimates back, we obtain
\begin{align*}
\frac{	|\mathcal{T}[	\tilde{F}_i,\nabla\mu,v^{\otimes^j}](t,x)|}{\|\rho\|_{1,T}M}&\lesssim\int_{0}^{t}\int_{s}^{t}s\int \langle s,x-(t-s)v\rangle^{-4+\kappa_0} \langle \tau,x-(t-\tau)v\rangle^{-3+\kappa_0}\frac{dv d\tau ds}{	\langle v\rangle^{4}}\\&\quad+\int_{0}^{t}\int_{s}^{t}\frac{s(\tau-s)(t-\tau+1)}{t-s}\int \langle s,x-(t-s)v\rangle^{-4+\kappa_0} \langle \tau,x-(t-\tau)v\rangle^{-4+\kappa_0}\frac{dv d\tau ds}{	\langle v\rangle^{3}}\\&:= A_7(t,x)+A_8(t,x).
\end{align*}
By \eqref{a4} and \eqref{a6}, we obtain 
\begin{equation*}
	|\nabla_{x}\mathcal{R}_j(\rho)(t,x)|\lesssim \langle t,x\rangle^{-4}\|\rho\|_{1,T}M.
\end{equation*}

Combining the bounds for $\mathcal{R}_j(\rho)$ and $\mathcal{I}_j(\rho)$, we conclude \eqref{A0}. This proves Proposition \ref{propoglobal3d}.
	\end{proof}

\section{Appendix}
In this section, we present some fundamental estimates used in the preceding sections.

\begin{lemma}[Basic Integral Estimates] Let $\beta_1\geq \max(4, \beta_2)\geq 3$  and  $n>0$. There holds 
	\begin{align}\label{a8}
		\int_{0}^{t}\int\frac{1}{\langle t-s,x-y\rangle^{\beta_1}} \frac{1}{\langle s,y\rangle^{\beta_2}} dyds\lesssim \mathbf{1}_{\beta_1>4} \frac{1}{\langle t,x\rangle^{\beta_2}}+\mathbf{1}_{\beta_1=4} \frac{\log(2+|t|+|x|)}{\langle t,x\rangle^{\beta_2}},
	\end{align}
	and 
		\begin{align}\label{a3}
		\int \langle x-tv,v\rangle^{-n-3} dv\lesssim \frac{1}{\langle t\rangle^3} \langle \frac{x}{\langle t\rangle}\rangle^{-n}.
	\end{align}
\end{lemma}
\begin{proof} Let's  start with the double integral
	$$	\int_{0}^{t}\int\frac{1}{\langle t-s,x-y\rangle^{\beta_1}} \frac{1}{\langle s,y\rangle^{\beta_2}} dyds.$$
	Decomposing the integration based on the relative size of $\langle t-s,x-y\rangle$ and $\langle t,y\rangle$, we consider three cases: 
	$$\langle t,x\rangle>2\langle s,y\rangle,\quad \langle t,x\rangle<\langle s,y\rangle/2,\quad  \langle t,x\rangle\sim\langle s,y\rangle.$$
	Summing the contributions from these cases, we obtain \eqref{a8}.

 For $t\in [0,1]$, $\langle x-tv,v\rangle^{-n-3}\sim \langle x,v\rangle^{-n-3}$. When  $t\geq 1$,  we can bound
\begin{align*}
	\int \langle x-tv,v\rangle^{-n-3} dv=\frac{1}{t^3}	\int \langle w,\frac{x-w}{t}\rangle^{-n-3} dw\lesssim \frac{1}{t^3}	\int \langle w,\frac{x}{t}\rangle^{-n-3} dw\lesssim  \frac{1}{t^3} \langle \frac{x}{t}\rangle^{-n}.
\end{align*}
So, we obtain \eqref{a3}. 
\end{proof}
\begin{lemma}[Decay Estimates for $A_j(t,x)$]\label{4.2} There holds 
	\begin{align}&
		A_1(t,x)+A_2(t,x)\lesssim \langle x,t\rangle^{-3},\label{a4}\\&
	\sum_{k=3}^{8}A_k(t,x)\lesssim \langle t\rangle\langle x,t\rangle^{-4},\label{a6}
\end{align}
for any $t\geq 1$ and $x\in \mathbb{R}^3$.
\end{lemma}
\begin{proof}  It is easy to check that
	\begin{align*}
\sum_{j}	\|A_j(t)\|_{L^\infty}\lesssim  t^{-3},
	\end{align*}
see \cite{HNX1}.
So, it is enough to prove that 
\begin{align*}
&A_1(t,x)+A_2(t,x)\lesssim \langle x\rangle^{-3},\\&
\sum_{k=3}^{8}A_k(t,x)\lesssim \langle t\rangle\langle x\rangle^{-4}
\end{align*}
for any  $|x|\geq t$ and $t\geq 1$. The proof follows by decomposing $A_{j}(t,x)$ into
\begin{align*}
		A_{j}(t,x)=	\int_0^{t/2}+
	\int_{t/2}^t=A_{j1}(t,x)+A_{j2}(t,x),
\end{align*}
for $j=1,\cdots,8.$

{\bf Step 1.} Estimate for $A_{j1}(t,x)$.\smallskip

We split the domain of $w$ based on its relative size compared to $x$, and using the bounds
\begin{align*}
	\int_{\R^3}\langle s,w\rangle^{-3+\kappa_0} \langle \tau,\frac{\tau-s}{t-s}x+\frac{t-\tau}{t-s}w\rangle^{-3+\kappa_0} dw\lesssim\langle s\rangle^{-\kappa_0} \langle \tau\rangle^{-3+3\kappa_0} (\frac{t-s}{t-\tau})^{2\kappa_0},
\end{align*}
we obtain 
\begin{align*}
A_{11}(t,x)&\lesssim	\int_0^{t/2}\int_{s}^{t}\int_{\R^3}\langle s,w\rangle^{-3+\kappa_0}\langle \frac{x}{t}\rangle^{-4}  \langle \tau,\frac{\tau-s}{t}x+\frac{t-\tau}{t}w\rangle^{-3+\kappa_0} \frac{dw d\tau ds}{t^3}\\&\quad+
\int_0^{t/2}\int_{s}^{t}\int_{\R^3}\mathbf{1}_{|x-w|\leq |x|/2}\langle s,x\rangle^{-3+\kappa_0}\langle \frac{x-w}{t}\rangle^{-4}  \langle \tau,x\rangle^{-3+\kappa_0} \frac{dw d\tau ds}{t^3}
\\
&\lesssim		\int_0^{t/2}\int_{s}^{t}\langle \frac{x}{t}\rangle^{-4}  \langle s\rangle^{-\kappa_0} \langle \tau\rangle^{-3+3\kappa_0} (\frac{t}{t-\tau})^{2\kappa_0}\frac{d\tau ds}{t^3}+\langle x\rangle^{-3}\lesssim\langle x\rangle^{-3}.
\end{align*}
Similarly, we also get 
\begin{align*}
		&A_{21}(t,x)
	\lesssim	\int_0^{t/2}\int_{s}^{t}(\tau-s)\int_{\R^3}\langle s,w\rangle^{-4+\kappa_0}\langle \frac{x}{t}\rangle^{-3}  \langle \tau,\frac{\tau-s}{t}x+\frac{t-\tau}{t}w\rangle^{-3+\kappa_0} \frac{dw d\tau ds}{t^3}\\&\quad+
	\int_0^{t/2}\int_{s}^{t}(\tau-s)\int_{\R^3}\mathbf{1}_{|x-w|\leq |x|/2}\langle s,x\rangle^{-4+\kappa_0}\langle \frac{x-w}{t}\rangle^{-3}  \langle \tau,x\rangle^{-3+\kappa_0} \frac{dw d\tau ds}{t^3}\\ \lesssim&
	\int_0^{t/2}\int_{s}^{t}\langle s\rangle^{-1+\kappa_0}\langle \frac{x}{t}\rangle^{-3}  \langle \tau\rangle^{-3+\kappa_0} \frac{d\tau ds}{t^3}+
	\int_0^{t/2}\int_{s}^{t}\langle s,x\rangle^{-4+\kappa_0}\langle \tau,x\rangle^{-2+2\kappa_0} \ln(2\langle x\rangle) d\tau ds
\lesssim \langle x\rangle^{-3},
\end{align*}
and 
\begin{align*}
	&A_{31}(t,x)+A_{51}(t,x)+A_{71}(t,x)
	\lesssim \int_0^{t/2}\int_{s}^{t}\int_{\R^3}\langle s,w\rangle^{-3+\kappa_0}\langle \frac{x}{t}\rangle^{-4}  \langle \tau,\frac{\tau-s}{t-s}x+\frac{t-\tau}{t-s}w\rangle^{-3+\kappa_0} \frac{dw d\tau ds}{t^3}\\&\quad\quad\quad\quad\quad\quad+
	\int_0^{t/2}\int_{s}^{t}\langle s,x\rangle^{-3+\kappa_0} \langle \tau,x\rangle^{-3+\kappa_0}  d\tau ds\\&\quad\quad\quad\quad
		\lesssim \int_0^{t/2}\int_{s}^{t}\langle \frac{x}{t}\rangle^{-4}  \langle s\rangle^{-\kappa_0} \langle \tau\rangle^{-3+3\kappa_0} (\frac{t}{t-\tau})^{2\kappa_0} \frac{ d\tau ds}{t^3}+
\langle t\rangle\langle x\rangle^{-5+2\kappa_0} \lesssim \langle t\rangle\langle x\rangle^{-4},\\
	&A_{61}(t,x)
	\lesssim \int_0^{t/2}\int_{s}^{t}(\tau-s)\int_{\R^3}\langle s,w\rangle^{-4+\kappa_0}\langle \frac{x}{t}\rangle^{-4}  \langle \tau,\frac{\tau-s}{t-s}x+\frac{t-\tau}{t-s}w\rangle^{-3+\kappa_0} \frac{dw d\tau ds}{t^3}\\&\quad\quad\quad\quad\quad\quad+
	\int_0^{t/2}\int_{s}^{t}(\tau-s)\langle s,x\rangle^{-4+\kappa_0} \langle \tau,x\rangle^{-3+\kappa_0}  d\tau ds\\&\quad\quad\quad\quad
	\lesssim \int_0^{t/2}\int_{s}^{t}(\tau-s)\langle \frac{x}{t}\rangle^{-4}  \langle s\rangle^{-1+\kappa_0} \langle \tau\rangle^{-3+\kappa_0}  \frac{ d\tau ds}{t^3}+
	\langle t\rangle\langle x\rangle^{-4} \lesssim \langle t\rangle\langle x\rangle^{-4}.
\end{align*}

Now we estimate $A_{41}(t,x),A_{81}(t,x)$. We have (write $\mathbf{1}_{(\tau-s)|x|\geq 2(t-\tau)|w|}+\mathbf{1}_{(\tau-s)|x|< 2(t-\tau)|w|}$ in the first iontegral)
\begin{align*}
&A_{41}(t,x)+A_{81}(t,x)
	\lesssim \int_0^{t/2}\int_{s}^{t}\tau(\tau-s)\int_{\R^3}\langle s,w\rangle^{-4+\kappa_0}\langle \frac{x}{t}\rangle^{-3}  \langle \tau,\frac{\tau-s}{t}x+\frac{t-\tau}{t}w\rangle^{-4+\kappa_0} \frac{dw d\tau ds}{t^3}\\&\quad+
	\int_0^{t/2}\int_{s}^{t}\tau (\tau-s)\int_{\R^3}\mathbf{1}_{|x-w|\leq |x|/2}\langle s,x\rangle^{-4+\kappa_0}\langle \frac{x-w}{t}\rangle^{-3}  \langle \tau,x\rangle^{-4+\kappa_0} \frac{dw d\tau ds}{t^3}\\&
		\lesssim \langle x\rangle^{-3}\int_0^{t/2}\int_{s}^{t}\tau(\tau-s)\int_{\R^3}\frac{t}{\tau-s}\langle x\rangle^{-1}\langle s,w\rangle^{-4+\kappa_0}  \langle \tau,\frac{\tau-s}{t}x+\frac{t-\tau}{t}w\rangle^{-3+\kappa_0} dw d\tau ds
\\&+\langle x\rangle^{-3}\int_0^{t/2}\int_{s}^{t}\tau(\tau-s)\int_{\R^3}
\frac{t-\tau}{\tau-s}\langle x\rangle^{-1}\langle s,w\rangle^{-3+\kappa_0} \langle \tau,\frac{\tau-s}{t}x\rangle^{-4+\kappa_0} dw d\tau ds
\\&\quad+	\int_0^{t/2}\int_{s}^{t}\tau (\tau-s)\log(|x|+2)\langle s,x\rangle^{-4+\kappa_0} \langle \tau,x\rangle^{-4+\kappa_0} d\tau ds\\&
\lesssim
	\langle x\rangle^{-4}\int_0^{t/2}\int_{s}^{t}t\tau\langle s\rangle^{-1+\kappa_0}  \langle \tau\rangle^{-3+\kappa_0}d\tau ds+
\langle x\rangle^{-4}\int_0^{t/2}\int_{s}^{t}t\tau\langle s\rangle^{-\kappa_0}  \langle \tau\rangle^{-4+3\kappa_0}d\tau ds
\\&\quad+\int_0^{t/2}\int_{s}^{t}\tau (\tau-s)\langle s,x\rangle^{-4+2\kappa_0} \langle \tau,x\rangle^{-4+\kappa_0} d\tau ds\lesssim \langle t\rangle
	\langle x\rangle^{-4}.
\end{align*}

{\bf Step 2.} Estimate for $A_{j2}(t,x)$. We have 
\begin{align*}
	&A_{12}(t,x)\lesssim	\int_{t/2}^t\int_{s}^{t}\int_{\R^3}\langle t,x\rangle^{-3+\kappa_0}\langle v\rangle^{-4}  \langle t,x\rangle^{-3+\kappa_0} dv d\tau ds
	\\&\quad\quad+
	\int_{t/2}^t\int_{s}^{t}\int_{\R^3}\langle t,x-(t-s)v\rangle^{-3+\kappa_0}\langle \frac{x}{t-s}\rangle^{-4}  \langle t,x-(t-\tau) v\rangle^{-3+\kappa_0} dv d\tau ds\\
	&\quad\lesssim \langle t,x\rangle^{-4+2\kappa_0}
	+
	\int_{t/2}^t\int_{s}^{t}\langle t\rangle^{-\kappa_0}(t-s)^{-3+2\kappa_0}\langle \frac{x}{t-s}\rangle^{-4}  \langle t\rangle^{-3+3\kappa_0} (t-\tau)^{-2\kappa_0}  d\tau ds
	\lesssim \langle x\rangle^{-3},\\&
	A_{22}(t,x)\lesssim	\int_{t/2}^t\int_{s}^{t}(\tau-s)\int_{\R^3}\mathbf{1}_{|v|\leq \frac{|x|}{2(t-s)}}\langle t,x\rangle^{-4+\kappa_0}\langle v\rangle^{-3}  \langle t,x\rangle^{-3+\kappa_0} dv d\tau ds
	\\&\quad\quad+
	\int_{t/2}^t\int_{s}^{t}(\tau-s)\int_{\R^3}\langle t,x-(t-s)v\rangle^{-4+\kappa_0}\langle \frac{x}{t-s}\rangle^{-3}  \langle t,x-(t-\tau) v\rangle^{-3+\kappa_0} dv d\tau ds\\
	&\quad\lesssim \langle x\rangle^{-3}
	+|x|^{-3}
	\int_{t/2}^t\int_{s}^{t}(\tau-s)\langle t\rangle^{-\kappa_0}(t-s)^{-1+2\kappa_0} \langle t\rangle^{-3+3\kappa_0} (t-\tau)^{-2\kappa_0}  d\tau ds\lesssim \langle x\rangle^{-3}.
\end{align*}

Similarly, we also get 
\begin{align*}
&	A_{32}(t,x)+	A_{52}(t,x)+	A_{72}(t,x)\lesssim	\int_{t/2}^t\int_{s}^{t}\int_{\R^3}\langle t,x\rangle^{-3+\kappa_0}\langle v\rangle^{-4}  \langle t,x\rangle^{-3+\kappa_0} dv d\tau ds
	\\&\quad\quad\quad\quad+
	\int_{t/2}^t\int_{s}^{t}\int_{\R^3}\langle t,x-(t-s)v\rangle^{-3+\kappa_0}\langle \frac{x}{t-s}\rangle^{-4}  \langle t,x-(t-\tau) v\rangle^{-3+\kappa_0} dv d\tau ds\lesssim \langle t\rangle\langle x\rangle^{-4},\\&
		A_{62}(t,x)\lesssim	\int_{t/2}^t\int_{s}^{t}(s-\tau)\int_{\R^3}\langle t,x\rangle^{-4+\kappa_0}\langle v\rangle^{-4}  \langle t,x\rangle^{-3+\kappa_0} dv d\tau ds
	\\&\quad+
	\int_{t/2}^t\int_{s}^{t}(s-\tau)\int_{\R^3}\langle t,x-(t-s)v\rangle^{-4+\kappa_0}\langle \frac{x}{t-s}\rangle^{-4}  \langle t,x-(t-\tau) v\rangle^{-3+\kappa_0} dv d\tau ds\lesssim \langle t\rangle\langle x\rangle^{-4},
\end{align*}
and 
\begin{align*}
	&A_{42}(t,x)\lesssim	\int_{t/2}^t\int_{s}^{t}(\tau-s)\tau\int_{\R^3}\mathbf{1}_{|v|\leq \frac{|x|}{2(t-s)}}\langle t,x\rangle^{-4+\kappa_0}\langle v\rangle^{-3}  \langle t,x\rangle^{-4+\kappa_0} dv d\tau ds
	\\&\quad+|x|^{-3}
	\int_{t/2}^t\int_{s}^{t}(\tau-s)\tau\int_{\R^3}\langle t,x-(t-s)v\rangle^{-4+\kappa_0}(t-s) ^3 \langle t,x-(t-\tau) v\rangle^{-4+\kappa_0} dv d\tau ds
	\\
	&\lesssim \langle t\rangle \langle x\rangle^{-4}
\\&\quad+|x|^{-3}
	\int_{t/2}^t\int_{s}^{t}(\tau-s)\tau\int_{\R^3}\mathbf{1}_{\frac{\tau-s}{t-s}|x|\leq 2(t-\tau) |v|}\langle t,(t-s)v\rangle^{-4+\kappa_0}(t-s) ^3 \langle t,\frac{\tau-s}{t-s}x-(t-\tau) v\rangle^{-4+\kappa_0} dv d\tau ds\\&\quad
	+|x|^{-3}
	\int_{t/2}^t\int_{s}^{t}(\tau-s)\tau\int_{\R^3}\langle t,(t-s)v\rangle^{-4+\kappa_0}(t-s) ^3 \langle t,\frac{\tau-s}{t-s}x\rangle^{-4+\kappa_0} dv d\tau ds
	\\
	&\lesssim \langle t\rangle \langle x\rangle^{-4}
	+|x|^{-3}
	\int_{t/2}^t\int_{s}^{t}(\tau-s)\tau\langle t,\frac{\tau-s}{t-\tau}x\rangle^{-4+\kappa_0}(t-s) ^3 \langle t\rangle^{-1+\kappa_0} (t-\tau)^{-3} d\tau ds\\&\quad
	+|x|^{-3}
	\int_{t/2}^t\int_{s}^{t}(\tau-s)\tau\langle t\rangle^{-1+\kappa_0} \langle t,\frac{\tau-s}{t-s}x\rangle^{-4+\kappa_0}  d\tau ds
		\\
	&\lesssim \langle t\rangle \langle x\rangle^{-4}
	+|x|^{-3}
	\int_{t/2}^t\int_{s}^{t}(\tau-s)\tau \mathbf{1}_{|\tau-s|\geq |t-\tau|} (\frac{\tau-s}{t-\tau}|x|)^{-4+\kappa_0}(t-s) ^3 \langle t\rangle^{-1+\kappa_0} (t-\tau)^{-3} d\tau ds\\&\quad
		+|x|^{-4}
	\int_{t/2}^t\int_{s}^{t}\tau\mathbf{1}_{|\tau-s|<|t-\tau|}\langle t\rangle^{-3+\kappa_0}(t-s) ^3 \langle t\rangle^{-1+\kappa_0} (t-\tau)^{-2} d\tau ds\\&\quad
	+|x|^{-4}
	\int_{t/2}^t\int_{s}^{t}(t-s)\tau\langle t\rangle^{-1+\kappa_0} \langle t\rangle^{-3+\kappa_0}  d\tau ds
	\lesssim \langle t\rangle \langle x\rangle^{-4}.
\end{align*}
To estimate $A_{82}(t,x)$, we divide the integration $\int_{t/2}^{t}$ into $\int_{t/2}^{t}\mathbf{1}_{t-s\geq 1/2}+\int_{t/2}^{t}\mathbf{1}_{t-s< 1/2}$ and get 
\begin{align*}
&A_{82}(t,x)\lesssim 	A_{42}(t,x)+\int_{t-\frac{1}{2}}^{t}\int_{s}^{t}s\int \langle t,x-(t-s)v\rangle^{-4+\kappa_0} \langle t,x-(t-\tau)v\rangle^{-4+\kappa_0}\frac{dv d\tau ds}{	\langle v\rangle^{3}}\\&\lesssim 
 \langle t\rangle \langle x\rangle^{-4}
 +\langle t\rangle \int_{t-1/2}^t\int_{s}^{t}\int_{\R^3}\mathbf{1}_{|v|\leq \frac{|x|}{2(t-s)}}\langle t,x\rangle^{-4+\kappa_0}\langle v\rangle^{-3}  \langle t,x\rangle^{-4+\kappa_0} dv d\tau ds
 \\&\quad+\langle t\rangle |x|^{-3}
 \int_{t-1/2}^t\int_{s}^{t}\int_{\R^3}\langle t,x-(t-s)v\rangle^{-4+\kappa_0}(t-s) ^3 \langle t,x-(t-\tau) v\rangle^{-4+\kappa_0} dv d\tau ds
 \\&\lesssim 
 \langle t\rangle \langle x\rangle^{-4}
\\&\quad +\langle t\rangle \langle x\rangle^{-3}
 \int_{t-1/2}^t\int_{s}^{t}\int_{\R^3}\mathbf{1}_{\frac{\tau-s}{t-s}|x|\leq 2(t-\tau) |v|}\langle t,(t-s)v\rangle^{-4+\kappa_0}(t-s) ^3 \langle t,\frac{\tau-s}{t-s}x-(t-\tau) v\rangle^{-4+\kappa_0} dv d\tau ds\\&\quad
 +\langle t\rangle \langle x\rangle^{-3}
 \int_{t-1/2}^t\int_{s}^{t}\int_{\R^3}\langle t,(t-s)v\rangle^{-4+\kappa_0}(t-s) ^3 \langle t,\frac{\tau-s}{t-s}x\rangle^{-4+\kappa_0} dv d\tau ds\\&\lesssim  \langle t\rangle \langle x\rangle^{-4}
 +\langle t\rangle \langle x\rangle^{-3}	\int_{t-1/2}^t\int_{s}^{t}\langle t,\frac{\tau-s}{t-\tau}x\rangle^{-4+\kappa_0}(t-s) ^3 \langle t\rangle^{-1+\kappa_0} (t-\tau)^{-3} d\tau ds\\&\quad
  +\langle t\rangle \langle x\rangle^{-4}
 \int_{t-1/2}^t\int_{0}^{|x|}\int_{\R^3}\langle t,(t-s)v\rangle^{-4+\kappa_0}(t-s) ^4 \langle t,\tau'\rangle^{-4+\kappa_0} dvd\tau' ds
 \\&\lesssim  \langle t\rangle \langle x\rangle^{-4}
+\langle t\rangle \langle x\rangle^{-3}	\int_{t-1/2}^t\int_{s}^{t}\mathbf{1}_{|\tau-s|<|t-\tau|}\langle t,\frac{\tau-s}{t-s}x\rangle^{-4+\kappa_0}(t-s) ^3 \langle t\rangle^{-1+\kappa_0} (t-\tau)^{-3} d\tau ds\lesssim  \langle t\rangle \langle x\rangle^{-4}.
\end{align*}

This completes the proof of Lemma \ref{4.2}.
\end{proof}

\section* {Acknowledgments}

Q. H.  Nguyen  is supported by the Academy of Mathematics and Systems Science, Chinese Academy of Sciences startup fund, and the National Natural Science Foundation of China (No. 12050410257 and No. 12288201) and  the National Key R$\&$D Program of China under grant 2021YFA1000800.  He also wants to thank  Alexandru Ionescu
		 for his stimulating comments and suggestion to consider the Vlasov-Poisson system. D. Wei is partially supported by the National Key R\&D Program of China under the grant 2021YFA1001500. Z. Zhang is partially supported by NSF of China under Grant 12288101.

%

	\end{document}